\documentclass[11pt, reqno]{amsart}
\usepackage{amssymb, amstext, amscd, amsmath, amsthm}
\usepackage{color}
\usepackage{hyperref} 

\usepackage{xy}
\xyoption{all}

\theoremstyle{plain}
\newtheorem{theorem}{Theorem}[section]
\newtheorem{thm}[theorem]{Theorem}
\newtheorem{cor}[theorem]{Corollary}
\newtheorem{prop}[theorem]{Proposition}
\newtheorem{lem}[theorem]{Lemma}
\theoremstyle{definition}

\newtheorem{defn}[theorem]{Definition}

\newtheorem{eg}[theorem]{Example}

\newcommand{\bN}{{\mathbb{N}}}

\newcommand{\bT}{{\mathbb{T}}}

\newcommand{\bZ}{{\mathbb{Z}}}

  \newcommand{\A}{{\mathcal{A}}}

  \newcommand{\G}{{\mathcal{G}}}

\newcommand{\M}{{\mathcal{M}}}

  \newcommand{\Z}{{\mathcal{Z}}}

\newcommand{\fA}{{\mathfrak{A}}}
\newcommand{\fB}{{\mathfrak{B}}}

\newcommand{\fD}{{\mathfrak{D}}}

\newcommand{\fF}{{\mathfrak{F}}}

\newcommand{\Bp}{{\mathbf{p}}}
\newcommand{\Bq}{{\mathbf{q}}}

\newcommand{\rC}{{\mathrm{C}}}

\newcommand{\upchi}{{\raise.35ex\hbox{\ensuremath{\chi}}}}

\newcommand{\qforal}{\quad\text{for all}\quad}

\newcommand{\Aut}{\operatorname{Aut}}

\newcommand{\spn}{\operatorname{span}}

\newcommand{\Per}{\operatorname{Per}}

\newcommand{\ca}{\mathrm{C}^*}

\newcommand{\mt}{\varnothing}

\newcommand{\ol}{\overline}

\begin{document}
\title[Periodic higher rank graphs]{Periodic higher rank graphs revisited}
\author[D. Yang]{Dilian Yang}
\address{Dilian Yang,
Department of Mathematics $\&$ Statistics, University of Windsor, Windsor, ON
N9B 3P4, CANADA}
\email{dyang@uwindsor.ca}

\begin{abstract}
Let $P$ be a finitely generated cancellative abelian monoid. 
A $P$-graph $\Lambda$ is a natural generalization of a higher rank graph.
A pullback of $\Lambda$ is constructed by pulling it back over a given monoid morphism to $P$,
while a pushout of $\Lambda$ is obtained by modding out its periodicity $\Per\Lambda$,
which is deduced from a natural equivalence relation on $\Lambda$. 
One of our main results in this paper shows that, for a class of higher rank graphs $\Lambda$, $\Lambda$ is isomorphic to the pullback of its pushout
via a natural quotient map,
and that its graph C*-algebra can be embedded into the tensor product of the graph C*-algebra of its pushout and $\ca(\Per\Lambda)$. 
As a consequence, its cycline C*-algebra generated by the standard generators with 
equivalent pairs is an abelian core (particularly a MASA). Along the way, we give an in-depth study on periodicity of $P$-graphs.
\end{abstract}

\subjclass[2010]{46L05, 47L30} \keywords{$k$-graph, $P$-graph, pullback, pushout, periodicity, MASA}
\thanks{The author partially supported by an NSERC grant.}

\date{}
\maketitle

\section{Introduction}

Let $P$ be a finitely generated cancellative abelian monoid. 
It was first suggested to study $P$-graphs in \cite{KumPask}, where
$\bN^k$-graphs (now known as \textit{$k$-graphs} or \textit{higher rank graphs}) were studied. But this has not been done until very recently. 
Based on some ideas in \cite{KumPask}, Carlsen, Kang, Shotwell and Sims 
study $P$-graphs, their pullbacks and associated C*-algebras in \cite{CKSS14}. Making use of them as a tool, 
they successfully describe the picture 
of the primitive idea space in the C*-algebra of a row-finite and source-free higher rank graph, which gives an analogue of 
(directed) graph C*-algebras in \cite{HS04}.

In this paper, we study pullbacks and pushouts of $P$-graphs, and we are particularly interested in 
their applications to periodic higher rank graphs. 
Let $P$ and $Q$ be two finitely generated cancellative abelian monoids,  $f:P\to Q$ be a monoid morphism,
and $\Gamma$ be a row-finite and source-free $Q$-graph. It is natural to construct a $P$-graph $f^*\Gamma$, which is called the pullback of $\Gamma$ via 
$f$, from $\Gamma$ and $f$. This idea first appeared in \cite{KumPask} and was further explored in \cite{CKSS14}.  
To construct pushouts is a little bit more involved; we make use of an idea from  \cite{DY09a, DY09b},
which has been substantially generalized to higher rank graphs in \cite{CKSS14}. 
Let $\Lambda$ be a row-finite and source-free $P$-graph.  Define a natural equivalence relation $\sim$ on $\Lambda$,
and then construct the periodicity $\Per\Lambda$ from $\sim$. 
Loosely speaking, under some suitable conditions the pushout $\Lambda/\sim$ of $\Lambda$
is a $\Bq(P)$-graph, which is obtained by modding out the periodicity of $\Lambda$. 
One key property we observe is that the pushout $\Lambda/\sim$ is aperiodic. 
This seems rather natural, but to prove it needs some care.
It is shown that $\Lambda$ is isomorphic to the pullback of the pushout $\Lambda/\sim$: 
$\Lambda\cong\Bq^*(\Lambda/\sim)$. In this case, $\ca(\Lambda)$ can be embedded into 
the tensor product of $\ca(\Lambda/\sim)$ with $\ca(\Per)$. Actually, it turns out that 
this embedding is the best one could possibly obtain. 
Along the way, we also have an in-depth study on periodicity of $P$-graphs.

An application of our results is given to the cycline subalgebra $\M$ of a higher rank graph. 
By definition, $\M$ is the sub-C*-algebra of $\ca(\Lambda)$ generated by the standard
generators with equivalent pairs in $\Lambda$. 
It turns out that $\M$ plays an important role in the structure of $\ca(\Lambda)$ due to the following Cuntz-Krieger uniqueness theorem 
in \cite{BNR14}: a representation of $\ca(\Lambda)$ is injective if and only if so is its restriction onto $\M$.  
It was asked in \cite{BNR14} if $\M$ is an abelian core of $\ca(\Lambda)$ \cite{NR12}. 
We answer this affirmatively in our case. That is, we show that $\M$  is a maximal abelian subalgebra (MASA) of $\ca(\Lambda)$, 
which is actually isomorphic to the tensor product of the canonical diagonal algebra $\fD_\Lambda$ with
$\ca(\Per\Lambda)$, and that there is also a faithful conditional expectation from $\ca(\Lambda)$ onto $\M$. 

This paper is motivated by \cite{BNR14, CKSS14} and strongly influenced by \cite{CKSS14}.
In Section \ref{S:Pgra}, some necessary background on $P$-graphs and the associated C*-algebras is briefly given. 
The main result of Section \ref{S:pullback} is an embedding theorem, which roughly says that
the C*-algebras of a class of $k$-graph C*-algebras, being pullbacks, 
can be embedded into tensor products of $P$-graph C*-algebras with commutative C*-algebras.
It turns out that this is the best one could possibly obtain in general. 
In Section \ref{S:Periodicity} we study periodicity of $P$-graphs in detail. 
We give some characterizations of aperiodicity, and provide a (sort of) concrete description for periodicity.
Our results generalize and unify what we have known for row-finite and source-free $k$-graphs in the literature. 
We finish off this section by showing that all pullbacks are periodic. 
We believe that the results of this section may be of independent interest and be very useful in future study.
We in Section \ref{S:PP} reverse the process of Section \ref{S:pullback}. Loosely speaking, we start with a $P$-graph $\Lambda$ whose periodicity 
$\Per\Lambda$ is a subgroup of the Grothendieck group $\G(P)$ of $P$. By modding out its periodicity, we obtain an aperiodic pushout, whose 
pullback is isomorphic to $\Lambda$. 
As an application of our results, in Section \ref{S:k-str}, we answer the questions posed in \cite{BNR14} in our case.

\subsection*{Notation and Conventions}

In this paper, 
\textbf{all monoids are assumed to be finitely generated, cancellative, and abelian.}
If $P$ is such a monoid, we also regarded it a category with one object. 
We use $\G(P)$ to denote the Grothendieck group of $P$, and we always embed $P$ into $\G(P)$.

As usual, for $m, n\in \bZ^k$, we use 
 $m\vee n$ and $m\wedge n$ to denote the coordinate-wise maximum and minimum of $m$ and $n$, respectively. For $n\in \bZ^k$,
 we let $n_+=n\vee 0$ and $n_-=-(n\wedge 0)$.  Of course, $n=n_+-n_-$ with $n_+\wedge n_-=0$. 

\section{$P$-Graphs}
\label{S:Pgra}

Let $P$ be a monoid.
$P$-graphs are a generalization of $k$-graphs, and share many properties with $k$-graphs. In this section, we 
briefly recall some basics on $P$-graphs which will be needed later. Refer to \cite[Section 2]{CKSS14} for more details. 

A \textit{$P$-graph} is a countable small category $\Gamma$ with a functor $d:\Gamma\to P$ such that the following 
factorization property holds:
whenever $\xi\in\Gamma$ satisfies $d(\xi)=p+q$, there are unique elements $\eta, \zeta\in\Gamma$ such that 
$d(\eta)=p$, $d(\zeta)=q$ and $\xi=\eta\zeta$. Clearly, any $k$-graph is an $\bN^k$-graph.
All notions on $k$-graphs can be generalized to $P$-graphs. For instance, 
for $p\in P$, let $\Gamma^p=d^{-1}(p)$, and so $\Gamma^0$ is the vertex set of $\Gamma$. 
There are source and range maps $s, r: \Gamma\to \Gamma^0$ such that $r(\xi)\xi s(\xi)=\xi$ for all $\xi\in\Gamma$.
For $v\in\Gamma^0$,
$v\Gamma=\{\xi\in\Gamma: r(\xi)=v\}$. 
We say that a $P$-graph $\Gamma$ is \textit{row-finite and source-free} if 
$0<|v\Gamma^p|<\infty$ for all $v\in\Gamma^0$ and $p\in P$. 

Let
\[
\Omega_P=\{(p,q)\in P\times P\mid q-p\in P\}. 
\]
Define $d,s,r:\Omega_P\to P$ by $d(p,q)=q-p$, $s(p,q)=q$, and $r(p,q)=p$.
It is shown in \cite[Example 2.2]{CKSS14} that $\Omega_P$ is a row-finite and source-free $P$-graph.

Let $\Lambda$ and $\Gamma$ be two $P$-graphs. A \textit{$P$-graph morphism} between $\Lambda$ and $\Gamma$ is a 
functor $x:\Lambda\to \Gamma$ such that $d_\Gamma(x(\lambda))=d_\Lambda(\lambda)$ for all $\lambda\in\Lambda$.  
The \textit{infinite path space of $\Gamma$} is defined as
\[
\Gamma^\infty=\{x:  \Omega_P\to \Gamma\mid x \text{ is a } P\text{-graph morphism}\}.
\]
For $x\in \Gamma^\infty$ and $p\in P$, there is a unique element $\sigma^p(x)\in \Gamma^\infty$ defined by
\[
\sigma^p(x)(q,r)=x(p+q,p+r).
\]
That is, $\sigma^p$ is a shift map on $\Gamma^\infty$. If $\mu\in\Gamma$ and $x\in s(\mu)\Gamma^\infty$, then $\mu x$ 
is defined to be the unique infinite path such that 
$\mu x(0,p)=\mu\cdot x(0,p-d(\mu))$ for any $p\in P$ with $p-d(\mu)\in P$. 
If $\sigma^p(x)=\sigma^q(x)$ for some $p\ne q\in P$, $x$ is said to be \textit{(eventually) periodic}.

\begin{defn}
A $P$-graph $\Gamma$ is said to be {periodic} if there is $v\in\Gamma^0$ such that every $x\in v\Gamma^\infty$ is periodic. 
Otherwise, $\Gamma$ is called {aperiodic}. 
\end{defn}

For a row-finite and source-free $P$-graph $\Gamma$, we associate to it a universal C*-algebra $\ca(\Gamma)$ as follows.

\begin{defn}
Let $\Gamma$ be a row-finite and source-free $P$-graph. A \textit{Cuntz-Krieger $\Gamma$-family} in a C*-algebra $\A$ is a family $\{S_\lambda:\lambda\in\Gamma\}$ in $\A$
such that 
\begin{itemize}
\item[(CK1)] $\{S_v:v\in\Gamma^0\}$ is a set of mutually orthogonal projections;
\item[(CK2)] $S_\mu S_\nu=S_{\mu\nu}$ whenever $s(\mu)=r(\nu)$;
\item[(CK3)] $S_\lambda^* S_\lambda=S_{s(\lambda)}$ for all $\lambda\in\Gamma$;
\item[(CK4)] $S_v=\sum_{\lambda\in v\Gamma^p}S_\lambda S_\lambda^*$ for all $v\in\Gamma^0$ and $p\in P$.
\end{itemize}

The \textit{$P$-graph C*-algebra $\ca(\Gamma)$} is the universal C*-algebra among Cuntz-Krieger $\Gamma$-families. We usually use $s_\lambda$'s to 
denote its generators. 
\end{defn}

It is known that
\[
\ca(\Gamma)=\ol\spn\{s_\mu s_\nu^*: \mu, \nu\in \Gamma\}.
\]
One important property is that 
$\ca(\Gamma)$ and the reduced C*-algebra $\ca_r(\G_\Gamma)$ are isomorphic, where $\G_\Gamma$ is the groupoid associated to $\Gamma$. 

By the universal property of $\ca(\Gamma)$, there is a natural gauge action 
$\gamma$ of the dual group $\widehat{\G(P)}$ of $\G(P)$ on $\ca(\Gamma)$ defined by
\[
\gamma_\chi(s_\lambda)=\chi(d(\lambda))s_\lambda\quad (\chi\in\widehat{\G(P)},\ \lambda\in\Gamma).
\]
Averaging over $\gamma$ gives a faithful conditional expectation $\Phi$ from $\ca(\Gamma)$ onto the fixed point algebra $\ca(\Gamma)^\gamma$. 
It turns out that  $\ca(\Gamma)^\gamma$ is an AF algebra and 
\[
\fF_\Gamma:=\ca(\Gamma)^\gamma=\ol\spn\{s_\mu s_\nu^*: d(\mu)=d(\nu)\}.
\]
The \textit{(canonical) diagonal algebra} $\fD_\Gamma$ of $\ca(\Gamma)$ is defined as 
\[
\fD_\Gamma=\ol\spn\{s_\mu s_\mu^*: \mu\in\Gamma\},
\]
which is the canonical MASA of $\fF_\Gamma$. 

Furthermore, as $k$-graphs, we also have the following two important uniqueness theorems for $\ca(\Gamma)$:
the gauge invariant uniqueness theorem \cite[proposition 2.7]{CKSS14} and the Cuntz-Krieger uniqueness theorem \cite[Corollary 2.8]{CKSS14}.

\medskip

Throughout the rest of this paper, 
\begin{center}
\textbf{all $P$-graphs are assumed to be row-finite and source-free.} 
\end{center}
An $\bN^k$-graph is simply called a $k$-graph as in the literature.

\section{Pullbacks and An Embedding Theorem}

\label{S:pullback}

We begin this section with the notion of pullbacks. 

Let $P$ and $Q$ be two monoids, and $f:P\to Q$ be a (monoid) homomorphism.  
If $(\Gamma, d_\Gamma)$ is a $Q$-graph, the \textit{pullback of $\Gamma$ via $f$} is the $P$-graph $(f^*\Gamma,d_{f^*\Gamma})$
defined as follows: 
$
f^*\Gamma=\{(\lambda,p):d_\Gamma(\lambda)=f(p)\}$ with 
$d_{f^*\Gamma}(\lambda, p)=p$, $s(\lambda, p)=s(\lambda)$ and $r(\lambda,p)=r(\lambda)$. The composition of two paths in $f^*\Gamma$
is given by $(\mu,p)(\nu,q)=(\mu\nu, p+q)$ if $s(\mu)=r(\nu)$.  

The pullback $f^*\Gamma$ defined above is indeed a $P$-graph. The proof is completely similar to that of \cite[Lemma 3.2]{CKSS14} where $P=\bN^k$.  
The following properties are easily derived from the very definition of pullbacks, and so their proofs are omitted here. 

\begin{lem}
\label{L:P-rowfinite}
Let $P$ and $Q$ be two monoids, and $f:P\to Q$ be a homomorphism. 
Suppose that $(\Gamma, d_\Gamma)$ is a $Q$-graph. Then we have the following.
\begin{itemize}
\item[(i)] $\Gamma^0=(f^*\Gamma)^0$. 

\item[(ii)] If $\Gamma$ is source-free, then so is $f^*\Gamma$. The converse holds if $f$ is surjective. 

\item[(iii)] If $f^*\Gamma$ is row-finite, then so is $\Gamma$. The converse holds if $f$ is surjective. 
\end{itemize}
\end{lem}

 \smallskip
 
In the remainder of this section, we focus on morphisms $f$ induced from group homomorphisms on $\bZ^k$. 
In this case, we can prove the following embedding theorem.

\begin{thm}
\label{T:tensor}
Let $H$ be a subgroup of $\bZ^k$, and $\Bq:\bZ^k\to \bZ^k/H$ be the quotient map. 
Suppose that $\Gamma$ is a $\Bq(\bN^k)$-graph. 
Then there is an injective C*-homomorphism from $\ca(\Bq^*\Gamma)$ to
$\ca(\Gamma) \otimes \ca(H)$.\footnote{Since the group C*-algebra $\ca(H)$ is abelian, it does not matter which C*-tensor product one chooses.} 
\end{thm}

\begin{proof}
Without loss of generality, we assume that $H$ has  rank $r\ge 1$. Then there is a basis $\{f_1,\ldots, f_r, \ldots, f_k\}$ 
of $\bZ^k$ and $d_1,\ldots, d_r\ge 1$ such that $d_i$ divides $d_{i+1}$ for all $1\le i\le r-1$ and 
$\{d_1f_1,\ldots, d_r f_r\}$ is a basis of $H$ \cite[Theorem 2.12]{RGS99}.
Furthermore, it follows from \cite[Theorem 2.14]{RGS99}  that $\bZ^k/H\cong \bZ_{d_1}\times\cdots\times \bZ_{d_r}\times\bZ^{k-r}$ under the 
isomorphism:
\begin{align*}
\varphi: \bZ^k/H&\to \bZ_{d_1}\times\cdots\times \bZ_{d_r}\times\bZ^{k-r},\\
 \big[\sum_{i=1}^k n_if_i\big] & \mapsto ([n_1]_{d_1}, \ldots, [n_r]_{d_r}, n_{r+1}, \ldots, n_k).
\end{align*}
So we identify $\bZ^k/H$ with $\bZ_{d_1}\times\cdots\times \bZ_{d_r}\times\bZ^{k-r}$ via $\varphi$ in what follows. 


Define a  `projection'  $\jmath$ from $\bZ^k/H$ onto the torsion-free part via 
\begin{align*}
&\jmath: \bZ^k/H\to \bZ^k, \\
&\quad ([n_1]_{d_1}, \ldots, [n_r]_{d_r}, n_{r+1}, \ldots, n_k)\mapsto (0, \ldots, 0, n_{r+1}, \ldots, n_k).
\end{align*}
Clearly, $\jmath$ is a homomorphism. Also $\jmath\circ \Bq$ is the projection from $\bZ^k$ onto the last $k-r$ factors: 
\[
\jmath\circ \Bq(n)=(0, \ldots, 0, n_{r+1}, \ldots, n_k)\quad \text{for all}\quad n\in \bZ^k.
\]

For $1\le i\le r$, let $h_i:=d_if_i$ and
$V_i$ be the unitary generators in the group C*-algebra $\ca(H)$ corresponding to $h_i$. 
Define an action $\Theta: \bT^k\to \Aut\big(\ca(\Gamma)\otimes \ca(H)\big)$ via
\begin{align}
\label{E:action}
\Theta_t(s_{\mu}\otimes V^n)=t^{\jmath\circ d(\mu)} s_{\mu}\otimes t^n V^n \quad \text{for all}\quad t\in\bT^k,
\end{align}
where $n=(n_1,\ldots, n_r,0,\ldots,0)\in\bZ^k$, 
and we use the multi-index notation
$V^n=\prod_{i=1}^r V_i^{n_i}$.
 One may think of  the action $\Theta$ as
the tensor product of a natural action of $H^\perp (\subset \bT^k)$ on $\ca(\Gamma)$ and the 
gauge action of $\widehat H$ on $\ca(H)$. 

Notice that, for any $(\mu,n)\in\Bq^*\Gamma$, one has  
\[
n-\jmath\circ d(\mu)=n-\jmath\circ \Bq(n)=(n_1, \ldots, n_r, 0, \ldots, 0).
\]
Then one can verify that $\big\{s_{\mu}\otimes  V^{n-\jmath\circ  d (\mu)}: d(\mu)=\Bq(n)\big\}$ is a Cuntz-Krieger $\Bq^*\Gamma$-family. 
To this end, let $t_{(\mu,n)}:=s_{\mu}\otimes  V^{n-\jmath\circ  d (\mu)}$. Obviously (CK1)  and (CK3) hold true. For (CK2), let $s(\mu)=s(\nu)$ and $(\mu,n)$,
$(\nu,m)\in\Bq^*\Gamma$. Then 
\begin{align*}
t_{(\mu,n)}t_{(\nu,m)}
&=(s_{\mu}\otimes  V^{n-\jmath\circ  d (\mu)})(s_{\nu}\otimes  V^{m-\jmath\circ  d (\nu)})\\
&=s_{\mu}s_{\nu}\otimes  V^{n+m-\jmath\circ  d (\mu)-\jmath\circ  d (\nu)}\\
&=s_{\mu\nu}\otimes  V^{n+m-\jmath\circ  (d (\mu)+d (\nu))}\\ 
&=t_{(\mu\nu,n+m)}
  =t_{(\mu,n)(\nu,m)},
\end{align*}
where the third ``=" used the property (CK2) for $\{s_\mu: \mu\in\Gamma\}$ and the property that $\jmath$ is a homomorphism. 
To verify (CK4), for $v\in (\Bq^*\Gamma)^0=\Gamma^0$ and $n\in\bN^k$, we compute 
\begin{align*}
\sum_{(\mu,n)\in v(\Bq^*\Gamma)^n} t_{(\mu,n)}t_{(\mu,n)}^*
&=\sum_{(\mu,n)\in v(\Bq^*\Gamma)^n} \big(s_{\mu}\otimes  V^{n-\jmath\circ  d (\mu)}\big)\big(s_{\mu}\otimes  V^{n-\jmath\circ  d (\mu)}\big)^*\\
&=\sum_{(\mu,n)\in v(\Bq^*\Gamma)^n} s_{\mu}s_{\mu}^*\otimes  I\\
&=\sum_{\mu\in v\Gamma^{\Bq(n)}}s_{\mu}s_{\mu}^*\otimes  I\\
&=s_v\otimes I
\end{align*}
due to property (CK4) for  $\{s_\mu: \mu\in\Gamma\}$.

By the universal property of $\ca(\Bq^*\Gamma)$, there is a (unique) *-homomorphism $\pi$ determined by  
\begin{align}
\nonumber 
\pi: \ca(\Bq^*\Gamma)&\to \ca(\Gamma)\otimes \ca(H),\\
\label{E:pi}
s_{(\mu,n)} &\mapsto s_{\mu}\otimes V^{n-\jmath\circ  d (\mu)},
\end{align}
where $(\mu, n)\in\Bq^*\Gamma$.
It is easy to check that $\pi$ is equivariant between the gauge action $\gamma$ of $\bT^k$ on $\ca(\Bq^*\Gamma)$ and the action $\Theta$ 
on $\ca(\Gamma)\otimes \ca(H)$ defined by \eqref{E:action}.  
In fact, for $t\in\bT^k$ and $(\mu,n)\in\Bq^*\Gamma$, we have 
\begin{align*}
\Theta_t\circ \pi(s_{(\mu,n)})
&=\Theta_t\big(s_{\mu}\otimes V^{n-\jmath\circ  d (\mu)}\big)\\
&=t^{\jmath\circ d(\mu)} s_{\mu}\otimes t^{n-\jmath\circ  d (\mu)} V^{n-\jmath\circ  d (\mu)}\\
&=t^ns_{\mu}\otimes V^{n-\jmath\circ  d (\mu)}\\
&=\pi\circ \gamma_t (s_{(\mu,n)}).
\end{align*}
By the gauge invariant uniqueness theorem of $k$-graphs \cite[Theorem 3.4]{KumPask}, $\pi$ is injective.
\end{proof}

One naturally wonders if $\pi$ defined in \eqref{E:pi} is also surjective. 
But, unfortunately, this is not the case in general, as the following simple example shows. 
Actually, Example \ref{Eg:Sims} in Section \ref{S:k-str} shows that Theorem \ref{T:tensor} is the best one could have. 

 \begin{eg}
 \label{Eg:1}
 Let $H=2\bZ$ and consider $\bN/H=\bZ_2$\,-graph $\Gamma$: $\Gamma^0=\{v\}$, and $\Gamma^1=\{e\}$. 
So $q^*\Gamma$ is $1$-graph: $q^*\Gamma^0=\{v\}$, and $(q^*\Gamma)^1=\{(e,1)\}$.
That is, $q^*\Gamma$ is a single-vertex directed graph with one edge.
So $\ca(q^*\Gamma)=\ca(H)=\rC(\bT)$, while $\ca(\Gamma)\otimes \ca(H)=\ca(\bZ_2)\otimes \ca(H)$. 

We will return to this example again in Section \ref{S:PP}.
\end{eg}

In some special cases, the above embedding could become an isomorphism. The following generalizes \cite[Corollary 3.5 (iii)]{KumPask}.

\begin{cor}
If $\bZ^k/H$ is torsion-free, then $\ca(\Bq^*\Gamma)\cong\ca(\Gamma) \otimes \ca(H)$.
\end{cor}

\begin{proof}
Keep the same notation as in the proof of Theorem \ref{T:tensor}.
By the definition of $\pi$ in \eqref{E:pi}, one can see that the range of $\pi$ is the C*-algebra 
\begin{align*}
\fA:=\pi(\ca(\Bq^*\Gamma))
&=\ca\big(s_\mu s_\nu^* \otimes V^n: d(\mu)-d(\nu)=q(n-\jmath\circ q(n))\big)\\
&=\ca\big(s_\mu s_\nu^* \otimes V^n: d(\mu)_i-d(\nu)_i=[n_i]_{d_i}: 1\le i\le r\big).
\end{align*}
Since $\bZ^k/H$ is torsion-free, then $d_i$=1 and $d(\mu)_i=d(\nu)_i=[n_i]=0$ for $1\le i\le r$. 
Thus $s_\mu s_\nu^*\otimes V_n$ is in $\fA$ for all $\mu, \nu\in\Gamma$ and $n\in H$. 
Therefore, $\fA=\ca(\Gamma) \otimes \ca(H)$, as desired. 
\end{proof}

One important consequence of Theorem \ref{T:tensor} is given below. 
But a lemma first.

\begin{lem}
\label{L:W} 
{\rm (\cite[Proposition 3.3]{CKSS14})}
Under the same conditions as in Theorem \ref{T:tensor}. Then, for any $h\in H$, there corresponds to a unitary $W_h$ in the centre of $\M(\ca(\Bq^*\Gamma))$,
the multiplier algebra of $\ca(\Bq^*\Gamma)$, 
which is given by 
\begin{align*}
\label{E:wi}
W_h
=\text{s-}\lim_{F}\sum_{v\in F} \sum_{\lambda\in v\Gamma^{\Bq({h}_+)}} \, s_{(\lambda,{h}_+)}s_{(\lambda,{h}_-)}^*.
\end{align*}
Here the limit is taking in the strict topology as $F$ increases over finite subsets of $\Gamma^0$.
Furthermore, one has 
\[
s_{(\lambda,{h}_+)}=W_h\, s_{(\lambda,{h}_-)}=s_{(\lambda,{h}_-)} W_h \qforal \lambda \in \Gamma^{\Bq({h}_+)}.
\]
\end{lem}

Keep the same notation as in the proof of Theorem \ref{T:tensor}. Suppose that the rank of $H$ is $r$ ($\ge 1$ WLOG) and 
$H=\langle h_i:=d_if_i\mid 1\le i\le r\rangle$. To simplify our writing, for $1\le i\le r$, we use $W_i:=W_{h_i}$ to denote the central unitaries determined
by $h_i$ in Lemma \ref{L:W}, and also use the multi-index notation
$W^n=\prod_{i=1}^r W_i^{n_i}$ for $n=(n_1,\ldots, n_r)\in \bZ^r$.

\begin{cor}
\label{C:tensor}
Under the same conditions as in Theorem \ref{T:tensor}. Then
\begin{itemize}
\item[(i)]
 $\fD_{\Bq^*\Gamma}\cong \fD_\Gamma$ and $\fF_{\Bq^*\Gamma}\cong \fF_\Gamma$,
 \item[(ii)]
$\fF_{\Bq^*\Gamma}\ca(W_h: h\in H) \cong \fF_{\Bq^*\Gamma}\otimes \ca(W_h:h\in H).$
\end{itemize}
\end{cor}

\begin{proof}
Keep the same notation as in the proof of Theorem \ref{T:tensor}. 

\smallskip
(i) This follows from the definition of $\pi$ (cf. \eqref{E:pi}) as 
$\pi(s_{(\mu, m)}s_{(\nu, m)}^*) = s_\mu s_\nu^*\otimes I$ for all $(\mu,m), (\nu,m)\in\Bq^*\Gamma$. 

(ii) 
Let $\fA:=\fF_{\Bq^*\Gamma}\ca(W_h: h\in H)$. We first claim that 
\begin{align*}
\pi(\fA) = \fF_\Gamma\otimes \ca(V_i^{d_i}:1\le i\le r).
\end{align*}
To this end, notice that   
\begin{align}
\nonumber
\pi\big(s_{(\lambda,{h_i}_+)}s_{(\lambda,{h_i}_-)}^*\big)
\nonumber
&=\big(s_{\lambda}\otimes V^{{{h_i}_+}-\jmath\circ  d(\lambda)}\big)\big(s_{\lambda}^*\otimes {(V^*)}^{{{h_i}_-}-\jmath\circ  d(\lambda)}\big)\\
\nonumber
&=\big(s_{\lambda}s_{\lambda}^*\big)\otimes V^{{{h_i}_+}-\jmath\circ  d(\lambda)- {{{h_i}_-}+\jmath\circ  d(\lambda)}}\\
\nonumber
&=s_{\lambda}s_{\lambda}^*\otimes V^{h_i}\\
\label{E:1}
&=s_{\lambda}s_{\lambda}^*\otimes V_i^{d_i}.
\end{align}
Now arbitrarily take a standard element  $s_\mu s_\nu^* \in \fF_\Gamma$ and $n=(n_1,\ldots, n_r,0,\ldots, 0)\in \bZ^k$.
Consider $s_{(\mu,m)}s_{(\nu,m)}^*W^n\in \fA$ where $\Bq(m)=d(\mu)(=d(\nu))$. 
We derive from Lemma \ref{L:W} and \eqref{E:1} that\footnote{Below we use the usual convention: for an operator $A$, $A^k=(A^*)^{-k}$ if $k<0$.}
\begin{align*}
\nonumber
&\pi\left(s_{(\mu,m)}s_{(\nu,m)}^*\, W^n\right)\\
\nonumber
&= \text{s-}\lim_F\sum_{v\in F} \left\{ \big(s_{\mu}\otimes W^{m-\jmath\circ  \Bq(m)}\big) \big(s_{\nu}\otimes W^{m-\jmath\circ  \Bq(m)}\big)^* 
       \prod_i \left(\sum_{\lambda\in v\Gamma^{\Bq({h_i}_+)}} \, 
      s_{\lambda}s_{\lambda}^* \otimes V_i^{d_i}\right)^{n_i}\right\}\\
\nonumber
&= \text{s-}\lim_F\sum_{v\in F} \left\{ (s_{\mu}s_\nu^*\otimes I) \prod_i \left(\sum_{\lambda\in v\Gamma^{\Bq({h_i}_+)}} \, 
      s_{\lambda}s_{\lambda}^*\right)^{n_i} \otimes V_i^{d_in_i}\right\}\\
\nonumber
&= \text{s-}\lim_F\sum_{v\in F}\left\{s_\mu s_\nu^*\prod_i \left(\sum_{\lambda\in v\Gamma^{\Bq({h_i}_+)}} 
      s_{\lambda}s_{\lambda}^*\right)^{n_i} \otimes V^{dn}\right\}\\
\nonumber
&=\left\{\text{s-}\lim_{F}\sum_{v\in F} s_{\mu} s_\nu^*\prod_i \left(\sum_{\lambda\in v\Gamma^{\Bq({h_i}_+)}} \, s_{\lambda}s_{\lambda}^*\right)^{n_i} \right\}
     \otimes V^{dn}\\
\nonumber
\label{E:piA}
&= s_{\mu} s_\nu^* \otimes V^{dn}.
\end{align*}
Here the last ``="  above used the fact that, for each $1\le i\le r$,
$\sum_{\lambda\in v\Gamma^{\Bq({h_i}_+)}} \, s_{\lambda}s_{\lambda}^*$ 
is strictly convergent to the identity of the multiplier algebra of $\ca(\Gamma)$.
This fact can be proved by a standard argument (cf. the proof of Proposition 3.3 of \cite{CKSS14}).
Therefore, we have proved our claim. 

Combining the injectivity of $\pi$ thanks to Theorem \ref{T:tensor} with $\fF_\Gamma\cong \fF_{\Bq^*\Gamma}$ from (i), 
one now has $\fA\cong \fF_{\Bq^*\Gamma}\otimes \ca(V_i^{d_i}:1\le i\le r)$.
From \eqref{E:1}, in the multiplier algebra $\M(\ca(\Bq^*\Gamma))$, one can verify that $\ca(W_i:1\le i\le r)\cong \ca(V_i^{d_i}:1\le i\le r)$. 
Hence $\fA\cong\fF_{\Bq^*\Gamma}\otimes \ca(W_h:h\in H)$, which ends the proof. 
\end{proof}

Let us remark that Corollary \ref{C:tensor} (i) can also be seen directly from the following identity, which generalizes the one given in Lemma \ref{L:W}.

\textit{If $(\mu, m)\in \Bq^*\Gamma$ and $n\in \bN^k$ such that $h:=n-m\in H$, then 
\[
s_{(\mu, n)}=W_h s_{(\mu,m)}=s_{(\mu,m)} W_h.
\]
}
In fact, assume that $h=h_+-h_-$. Then $n+h_-=m+h_+$. Let $\lambda\in\Gamma$ satisfying $r(\lambda)=s(\mu)$
and $d(\lambda)=h_-$. So $(\lambda, h_+)$ also belongs to $\Bq^*\Gamma$.  
Then by Lemma \ref{L:W}
\[
s_{(\mu,n)}s_{(\lambda, h_-)}=s_{(\mu\lambda,n+h_-)}=s_{(\mu\lambda,m+h_+)}=s_{(\mu,m)}s_{(\lambda, h_+)}=s_{(\mu,m)}W_h s_{(\lambda, h_-)}.
\]
Hence 
\[
s_{(\mu,n)}\sum_{\lambda\in s(\mu)\Gamma^{h_-}} s_{(\lambda, h_-)}s_{(\lambda, h_-)}^*=s_{(\mu,m)} W_h \sum_{\lambda\in s(\mu)\Gamma^{h_-}}  s_{(\lambda, h_-)}s_{(\lambda, h_-)}^*,
\]
which implies $s_{(\mu, n)}=W_h s_{(\mu,m)}$.

\smallskip

An important and interesting application of Theorem \ref{T:tensor} and Corollary \ref{C:tensor} will be exhibited in Section
\ref{S:k-str}, after we investigate pushouts of $P$-graphs.
But we need to study periodicity of $P$-graphs first. This is given in next section.

\section{Periodicity}
\label{S:Periodicity}

It is known that periodicity of higher rank graphs plays a very important role in the structure of their C*-algebras. 
See, e.g.,  \cite{CKSS14, DY09a, DY09b, HLRS14, LS10, Raeburn, RS07, Sho12}. 
This section provides an in-depth analysis on periodicity of $P$-graphs.
We give some natural characterizations of the aperiodicity defined in \cite{KumPask} (via infinite paths), in terms of the triviality of  
$\Per\Lambda$ (coinciding with our notion of aperiodicity) and/or local periodicities. One crucial observation is the interplay between the (global) periodicity of $\Lambda$ 
and its local periodicities.
Applying those results, we prove that all pullbacks are periodic. We believe that our results of this section may be of independent interest and 
be useful in studying $P$-graphs in the future. 

Let  $P$ be a monoid and $\Lambda$ be a $P$-graph. Define an equivalence relation $\sim$ on $\Lambda$ as follows:
\begin{align}
\label{D:sim}
\xi\sim\eta \Longleftrightarrow s(\xi)=s(\eta) \text{ and } \xi x=\eta x \text{ for all }x\in s(\xi)\Lambda^\infty.
\end{align}
If $\xi\sim\eta$, obviously one also has $r(\xi)=r(\eta)$ automatically. So $\sim$ respects sources and ranges. 

Associate to $\sim$ two important sets: the \textit{periodicity of $\Lambda$}
\[
\Per\Lambda=\big\{d(\xi)-d(\eta): \xi,\eta\in\Lambda, \ \xi\sim\eta\big\}\subseteq\G(P),
\]
and

\begin{align*}
\Lambda^0_{\Per}=\left\{
v\in\Lambda^0\mid
\begin{array}{l}
\text{any } \xi\in v\Lambda, \,p\in P \text{ with } d(\xi)-p\in \Per\Lambda\\
\Longrightarrow\text {there is }\eta\in v\Lambda^p \text{ such that }\xi\sim\eta
\end{array}\right\}.
\end{align*}
Here are some versions of  `local periodicity':
\begin{align*}
\Sigma_v&=\big\{(m,n)\in P\times P: \sigma^m (x)=\sigma^n (x)\text{ for all }x\in v\Lambda^\infty\big\}\quad (v\in \Lambda^0),\\
\Per_v&=\{m-n: (m,n)\in \Sigma_v\} \quad (v\in \Lambda^0),\\
\Sigma_\Lambda&=\bigcup_{v\in\Lambda^0} \Sigma_v.
\end{align*}

The above notion of periodicity, $\Per\Lambda$, was first introduced in \cite{DY09a, DY09b} to study the representation theory of single-vertex $k$-graph algebras. 
It has been substantially generalized in \cite{CKSS14} to all $k$-graphs (also cf. \cite{HLRS14}).
The set $\Lambda_{\Per}^0$ was used in \cite{CKSS14} to construct a subgraph of $\Lambda$, while $\Sigma_v$ is related to a notion of local periodicity of $\Lambda$ at $v$
in \cite{RS07, LS10, Sho12}.

\subsection{Characterizations of aperiodicity}

In this subsection, let us fix a monoid $P$ and a $P$-graph $\Lambda$. 

\begin{lem}
\label{L:uniqueness}
Let $\lambda, \mu$ in  $\Lambda$ be such that $\lambda\sim \mu$ and $d(\mu)=d(\nu)$. Then $\lambda=\mu$.   
\end{lem}

\begin{proof}
Fix $x\in s(\lambda)\Lambda^\infty$. Since $\lambda\sim \mu$, we have $\lambda x=\mu x$. This implies $\lambda=(\lambda x)(0,d(\lambda))
=(\lambda x)(0,d(\mu))=(\mu x)(0,d(\mu))=\mu$.
\end{proof}

\begin{lem}
\label{L:xmn}
Let $v\in \Lambda^0$ and $(m,n)\in\Sigma_v$. Then 
$x(m,m+n)\sim x(n,m+n)$ for all $x\in v\Lambda^\infty$.
\end{lem}

\begin{proof}
Let $x\in v\Lambda^\infty$, $\nu:=x(m,m+n)$, and $\mu=x(n,m+n)$. Obviously, $s(\mu)=s(\nu)$. 

In the sequel, we first show 
\begin{align}
\label{E:xm}
x(0,m)\nu=x(0,n)\mu.
\end{align}
Indeed, by the factorization property, there are $\nu'\in v\Lambda^n$ and $\mu'\in s(\nu')\Lambda^m$ such that 
\[
x(0,m+n)=x(0,m)\nu=\nu'\mu'.
\]
But
\[
\nu'\mu'=x(0,m+n)=x(0,n)\mu'
\]
implies 
$
\nu'=x(0,n).
$
So 
\[
x(0,m)\nu=x(0,n)\mu'.
\]
Thus 
$x(0,m)\nu\sigma^{m+n}(x)=x(0,n)\mu'\sigma^{m+n}(x)$.
One now has
$\nu\sigma^{m+n}(x)=\mu'\sigma^{m+n}(x)$
as $(m,n)\in\Sigma_v$. 
But $\sigma^m (x)=\nu\sigma^{m+n}(x)$ and $\sigma^m (x)=\sigma^n (x)$, one has 
$\sigma^n (x)=\mu'\sigma^{m+n}(x)$, implying 
\[
\mu'=(\sigma^n (x))(0,m)=x(n,n+m)=\mu.
\]
This proves \eqref{E:xm}.

Now let $y\in s(\mu)\Lambda^\infty$. Then it follows from $(m,n)\in\Sigma_v$ and \eqref{E:xm} that 
\[
\mu y=\sigma^n(x(0,n)\mu y)=\sigma^m(x(0,n)\mu y)=\sigma^m(x(0,m)\nu y)=\nu y.
\]
This proves $\mu\sim \nu$. 
\end{proof}

The following result describes the relation between the periodicity of infinite paths and that induced from the equivalence relation $\sim$ 
defined by \eqref{D:sim}. Notice that it was proved in \cite{CKSS14} under the condition that $\Lambda^0$ is a maximal tail.

\begin{thm}
\label{T:SHLambda}
Let $m,n\in P$. Then 
$(m,n)\in \Sigma_\Lambda \Longleftrightarrow m-n\in \Per\Lambda$.
\end{thm}

\begin{proof}
``$\Rightarrow$":
Let $(m,n)\in \Sigma_\Lambda$. Then $(m,n)\in \Sigma_v$ for some $v\in \Lambda^0$. 
By Lemma \ref{L:xmn}, there are $\mu,\nu\in\Lambda$ such that $d(\mu)=m$, $d(\nu)=n$, and $\mu\sim \nu$. 
Thus $m-n\in \Per\Lambda$. 

``$\Leftarrow$": Assume that  $m,n\in P$ such that $m-n\in \Per\Lambda$. So, by definition, there are $\mu,\nu\in\Lambda$ satisfying 
$\mu\sim\nu$ and $m-n=d(\mu)-d(\nu)$. WLOG, we assume that $m\ne n$. Set $m':=d(\mu)$ and $n':=d(\nu)$. 

Let $x\in s(\mu)\Lambda^\infty$. Then from $\mu\sim\nu$ and $m+n'=m'+n$, on one hand, we have
\begin{align*}
\mu  x(0,n)
&=(\mu x)(0, m'+n) = (\nu x)(0,m'+n)\\
&=  (\nu x)(0,m+n')=\nu x(0,m). 
\end{align*}
On the other hand, $\mu\sim\nu$ gives
\[
\mu x(0,n)\sigma^n (x)=\mu x=\nu x = \nu x(0,m)\sigma^m (x).
\]
Therefore, combining the above identities yields $\sigma^m (x)=\sigma^n (x)$. This implies $(m,n)\in \Sigma_{s(\mu)}\subseteq\Sigma_\Lambda$, as desired. 
\end{proof}

The first corollary below is immediate from the above theorem. 

\begin{cor}
\label{C:cong}
$(m,n)\in\Sigma_\Lambda \Longleftrightarrow (m-p, n-p)\in\Sigma_\Lambda$ for any $p\in \G(P)$ such that $m-p,n-p\in P$. 
\end{cor}

As expected, one has

\begin{cor}
\label{C:periodicity}
$\Per\Lambda=\cup_{v\in \Lambda^0}\Per_v$.
\end{cor}

\begin{proof}
Assume $m-n\in\Per\Lambda$ with $m,n\in P$. Then $(m,n)\in \Sigma_v$ for some $v\in \Lambda^0$ by Theorem \ref{T:SHLambda}. So $m-n\in\Per_v$, proving 
$\Per\Lambda\subseteq \cup_v \Per_v$.

Conversely, let $m,n\in P$ be such that $m-n\in\Per_v$ for some $v\in \Lambda^0$. Then by definition $m-n =m'-n'$ for some $(m',n')\in \Sigma_v$. 
Applying Theorem \ref{T:SHLambda} again yields $m-n\in \Per\Lambda$. Thus $\cup_v\Per_v\subseteq\Per\Lambda$. 
\end{proof}

Slightly different from \cite{RS07, LS10, Sho12}, we call $\Per_v$ the \textit{local periodicity of $\Lambda$ at $v$.} 
So Corollary \ref{C:periodicity} tells us that the (global) periodicity of $\Lambda$ is the union of its all local periodicities. 
If $\Per_v=\{0\}$ for all $v\in \Lambda^0$, then we say that \textit{$\Lambda$ has no local periodicity}. 

The most important consequence of Theorem \ref{T:SHLambda} is probably the characterizations of aperiodicity/periodicity given below, which unify and generalize
what we have known for row-finite and source free $k$-graphs in the literature. 

\begin{thm}
\label{T:cha}
Let $\Lambda$ be a $P$-graph. Then the following are equivalent:
\begin{itemize}
\item[(i)] $\Lambda$ is aperiodic;
\item[(ii)] $\Per\Lambda=\{0\}$;
\item[(iii)] $\Lambda$ has no local periodicity.
\end{itemize}
\end{thm}

\begin{proof}
(ii) $\Leftrightarrow$ (iii) follows from Corollary \ref{C:periodicity} immediately.

(i) $\Rightarrow$ (ii):  Suppose that $\Lambda$ is aperiodic. To the contrary, assume $\Per\Lambda\ne \{0\}$. Then 
there are $\mu\ne \nu\in \Lambda$ such that $\mu\sim \nu$. Since $\Lambda$ is aperiodic, there is an aperiodic $x\in s(\mu)\Lambda^\infty$.
But we also have $\mu x=\nu x$. So $\sigma^{d(\nu)}(x)=\sigma^{d(\mu)+d(\nu)}(\mu x)= \sigma^{d(\mu)+d(\nu)}(\nu x)=\sigma^{d(\mu)}(x)$. 
This contradicts with $x$ being aperiodic. 

(iii) $\Rightarrow$ (i): This can be proved after applying some obvious modifications to the proof of \cite[Lemma 3.2]{RS07} (iii) $\Rightarrow$ (i) (for example,
by replacing $m\vee n$ there with $m+n$). 
\end{proof}

One can also obtain some properties of $\Lambda_{\Per}^0$ by making use of Theorem \ref{T:SHLambda}. 

\begin{cor}
\label{C:Lambda0}
$\Lambda_{\Per}^0\subseteq\big\{v\in\Lambda^0: \Sigma_v=\Sigma_\Lambda\big\}$. 
\end{cor}

\begin{proof}
Let $(m,n)\in \Sigma_\Lambda$, $v\in \Lambda_{\Per}^0$, $x\in v\Lambda^\infty$, and $\mu=x(0,m)$. 
Since $v\in \Lambda_{\Per}^0$, due to Theorem \ref{T:SHLambda} there is $\nu\in v\Lambda$ such that $d(\nu)=n$ and $\mu\sim \nu$. 
Then one has 
\[
\mu \sigma^m (x)=\nu \sigma^m (x).
\] 
This implies $\nu=x(0,n)$. 
Thus 
\[
\mu\sigma^m (x) =x=x(0,n)\sigma^n (x)=\nu \sigma^n (x).
\] 
Comparing the above identities gives $\nu\sigma^m (x)=\nu\sigma^n (x)$, and so $\sigma^m (x)=\sigma^n (x)$.
Thus $(m,n)\in \Sigma_v$, proving the desired inclusion. 
\end{proof}

We do not known if $\Lambda_{\Per}^0=\big\{v\in\Lambda^0: \Sigma_v=\Sigma_\Lambda\big\}$ in general.
We expect so, but have not found a proof yet. 
The corollary below is quite strong; but notice that
all strongly connected $k$-graphs studied in \cite{HLRS14} satisfy its requirements. 

\begin{cor}
\label{C:Lambda0}
If $\Lambda$ is sink-free and $\Sigma_v=\Sigma_\Lambda$ for every $v\in\Lambda^0$, then 
$\Lambda_{\Per}^0=\Lambda^0$. 
\end{cor}

\begin{proof}
Let $v\in\Lambda^0$. Suppose $\mu\in v\Lambda^m$ and $n\in\bN^k$ satisfies $m-n\in \Per\Lambda$. By Theorem \ref{T:SHLambda},
$(m,n)\in\Lambda_w=\Sigma_\Lambda$ for all $w\in\Lambda^0$. 
Arbitrarily choose $x\in \Lambda^\infty$ such that $\mu=x(n, m+n)$. This can be done as $\Lambda$ is also sink-free. 
Since $\Sigma_{r(x)}=\Sigma_\Lambda$, one also gets $(m,n)\in\Sigma_{r(x)}$. Applying Lemma \ref{L:xmn} gives $\mu\sim x(m,m+n)$. 
Thus $v\in \Lambda_{\Per}^0$; we are done. 
\end{proof}

\subsection{Pullbacks are periodic}

In this subsection, we apply the main results of the previous one to prove that all pullbacks are always periodic. 

Let $\Gamma$ be a $P$-graph. Notice that $x\in\Gamma^\infty$ is uniquely determined by $x(0,p)$ for all $p\in P$. Indeed, for $p,q\in P$ with $q-p\in P$, then 
$x(p,q)$ is unique determined by the (unique) factorization $x(0,q)=x(0,p)x(p,q)$.

The following result is in the same vein with \cite[Proposition 2.9]{KumPask} and \cite[Proposition 2.4]{BNR14}. 

\begin{prop}
\label{P:dotx}
Let $P$, $Q$ be two monoids, $f:P\to Q$ be a surjective homomorphism, and $\Gamma$ be a $Q$-graph.
Then $f$ induces a homeomorphism
 $f_*: (f^*\Gamma)^\infty\to \Gamma^\infty$ by 
\begin{align*}
f_*(x)(0,f(n))=\Bp_1(x(0,n))\qforal x\in(f^*\Gamma)^\infty,\  n\in P,
\end{align*}
where $\Bp_1:f^*\Gamma\to \Gamma$ is the projection $\Bp_1(\mu,n)=\mu$ for all $(\mu,n)\in f^*\Gamma$. 
\end{prop}

\begin{proof} 
Since $f$ is a surjective homomorphism,
observe that 
\begin{align*}
\Omega_{Q}
=\{(0, f(n)): n \in P\}.
\end{align*}

We now prove that $f_*$ is well-defined. 
That is, we need to show the following: 
if $n,n'\in P$ satisfy  $f(n)=f(n')$,
then
\begin{align*}
\label{E:xmn}
\Bp_1(x(0,n)) = \Bp_1(x(0,n')) \qforal x\in(f^*\Gamma)^\infty.
\end{align*}

Notice that 
$x(0,n)x(n,n+n')=x(0,n')x(n',n+n')$ implies 
\[
\big(\Bp_1(x(0,n)), n)\big)\big(\Bp_1(x(n,n+n')),n'\big)=\big(\Bp_1(x(0,n')), n'\big)\big(\Bp_1(x(n',n+n')),n\big).
\] 
Taking $\Bp_1$ on both sides gives 
\[
\Bp_1(x(0,n))\Bp_1(x(n,n+n'))=\Bp_1(x(0,n'))\Bp_1(x(n',n+n')).
\] 
But $d(\Bp_1(x(0,n)))=f(n)=f(n')=d(\Bp_1(x(0,n'))$. So one has $\Bp_1(x(0,n))=\Bp_1(x(0,n'))$ by the factorization property of $\Gamma$. 

By the discussion before the statement of this proposition, $f_*(x)\in\Gamma^\infty$ for any $x\in\Lambda^\infty$.
The rest of the proof is similar to \cite[Proposition 2.4]{BNR14}. 
\end{proof}

The identity in the following lemma turns out to be very handy.

\begin{lem}
\label{L:x}
Let $P$, $Q$ be two monoids, $f:P\to Q$ be a homomorphism, and $\Gamma$ be a $Q$-graph.
If $m,n\in P$ such that $f(m)=f(n)$, then 
\[
(\mu,m)(w,p)(\nu,n)=(\mu,n)(w,p)(\nu,m)
\]
for all $(\mu,m)$, $(w,p)$, $(\nu,n)\in f^*\Gamma$ with $w\in s(\mu)\, \Gamma\, r(\nu)$.
\end{lem}

\begin{proof}
Let  $(\mu,m)$, $(w,p)$, $(\nu,n)\in f^*\Gamma$ with $w\in s(\mu)\, \Gamma\, r(\nu)$. 
Then the factorization property of $f^*\Gamma$ gives 
$(\mu,m)(w,p)(\nu,n)=(\nu',n)(w',p)(\mu',m)$ for some $\mu', \nu', w'$ with $w'\in s(\nu')\,\Gamma\, r(\mu')$. 
Then taking $\Bp_1$ at both sides yields 
\[
\mu w \nu=\nu' w' \mu'\in\Gamma.
\] 
But $d(\mu)=f(m)=f(n)=d(\nu')$, and 
similarly $d(\nu)=d(\mu')$. Clearly $d(w)=d(w')$. 
By the unique factorization property of $\Gamma$, 
one has 
\[
\nu'=\mu, \ w'=w,  \ \mu'=\nu, 
\]  
as desired.
\end{proof}

\begin{thm}
\label{P:HP}
Let $P$ be a monoid, $H$ be a non-zero subgroup of $\G(P)$, $\Bq: \G(P)\to \G(P)/H$ be the quotient map, and $\Gamma$ be a $P/H$-graph. 
Then $\Bq^*\Gamma$ is a periodic $P$-graph.
\end{thm}

\begin{proof}
In what follows, we show a little bit more than what is needed: if $m$, $n$ in $P$ such that $m-n\in H$, then 
\[
\sigma^m(x)=\sigma^n(x)\quad \text{for all}\quad x\in \Lambda^\infty.
\]
Once this is done, we obtain $H\subseteq \Per\Bq^*\Gamma$ by Theorem \ref{T:SHLambda}. Thus $\Bq^*\Gamma$ is periodic by Theorem \ref{T:cha}. 

To this end, let $a_1,\ldots, a_k$ be fixed generators of $P$ (recall that $P$ is finitely generated).
Fix $m,n\in P$ such that $m-n\in H$. Then choose $\ell\in P$ satisfying $m+ n+ \ell - (a_1+\cdots +a_k)\in P$.
(The existence of such $\ell$ is clear.)
Let $x\in\Lambda^\infty$. Since $\Lambda$ is source-free, one can always write $x$ as  
\begin{align*}
x
&=(\mu_1, m)(\nu_1, n)(w_1, \ell)(\mu_2, m)(\nu_2, n)(w_2,\ell)\cdots\\
&=\prod_{i=1}^\infty (\mu_i, m)(\nu_i, n)(w_i,\ell).
\end{align*}
where $\mu_i,\nu_i,w_i\in\Gamma$ with $d(\mu)=\Bq(m)$, $d(\nu)=\Bq(n)$, $d(w_i)=\Bq(\ell)$.

Then on one hand
\[
\sigma^m(x)=(\nu_1, n)(w_1, \ell)(\mu_2, m)(\nu_2, n)(w_2,\ell)\cdots=\prod_{i\ge 2} (\nu_{i-1}, n)(w_{i-1}, \ell) (\mu_i, m).
\]
On the other hand, since $m-n\in H$, one repeatedly applies Lemma \ref{L:x} to obtain
\begin{align*}
\sigma^n(x)
&=\sigma^n((\mu_1, n)(\nu_1, m)(w_1, \ell)(\mu_2, m)(\nu_2, n)(w_2,\ell)\cdots) \ (\text{by Lemma }\ref{L:x})\\
&=(\nu_1, m)(w_1, \ell)\prod_{i\ge 2}(\mu_i, m)(\nu_i, n)(w_i,\ell)\\ 
&=(\nu_1, m)(w_1, \ell)\prod_{i\ge 2} (\mu_i, n)(\nu_i, m) (w_i,\ell) \ (\text{by Lemma }\ref{L:x})\\
&=(\nu_1, m)(w_1, \ell)(\mu_2,n)\cdot (\nu_2,m) (w_2,\ell) (\mu_3,n)\cdot \cdots \ (\text{by rearrangement}) \\ 
&=(\nu_1, n)(w_1, \ell)(\mu_2,m)\cdot (\nu_2,n) (w_2,\ell) (\mu_3,m)\cdot \cdots \ (\text{by Lemma } \ref{L:x}) \\
&=\prod_{i\ge 2} (\nu_{i-1}, n)(w_{i-1}, \ell) (\mu_i, m).
\end{align*}
Therefore $\sigma^m(x)=\sigma^n(x)$. 
\end{proof}

It is probably worth remarking that, in general, $H\subsetneq \Per\Bq^*\Gamma$. 
For example, let $H=2\bZ$ and $\Gamma$ be the $\bZ_2$-graph consisting of one vertex $v$ and one edge $e$. Then $\Per{\Bq^*\Gamma}=\bZ$.


\section{Pushouts}

\label{S:PP}

In this section, we reverse the pulling back process of Section \ref{S:pullback}:
we start with a $P$-graph $\Lambda$, and construct an aperiodic graph from $\Lambda$ by modding out its periodicity. 
In order to achieve such a goal, we need that $\Per\Lambda$ is a subgroup of $\G(P)$. To this end, 
we impose the condition that $\Lambda^0$ has property \textsf{W}, which is one key property required in the notion of 
a maximal tail (cf. \cite{CKSS14, KP11}). 

\begin{defn}
We say $\Lambda^0$ has property \textsf{W} if 
for any $u, v\in \Lambda^0$ there is $w\in \Lambda^0$ such that $u\Lambda w\ne \mt$ and $v\Lambda w\ne \mt$.
\end{defn}

The name ''property \textsf{W}" is self-explanatory: the vertices $u,v$ are contained in a (directed) wedge.
It is very easy to check that if a $P$-graph $\Lambda$ is cofinal, then $\Lambda^0$ has property \textsf{W}. It is also noteworthy that all strongly connected graphs are cofinal
and so have property \textsf{W}.

Recall that the equivalence relation $\sim$ defined in \eqref{D:sim} respects sources and ranges, and so $\Lambda/\sim$ is a category: for all $\xi, \eta\in \Lambda$
\[
r([\xi])=[r(\xi)], \ s([\xi])=[s(\xi)], \ [\xi][\eta]=[\xi\eta].
\]

The following theorem is an analogue of \cite[Theorem 4.2]{CKSS14}, which is proved under the condition that $\Lambda^0$ is a maximal tail, but 
only property \textsf{W} is essentially used. Actually, Lemma 4.5 and Corollary 4.6 in \cite{CKSS14} have been generalized to Corollary 
\ref{C:cong} and Theorem \ref{T:SHLambda} in Section \ref{S:Periodicity} above to all $P$-graphs.

\begin{thm}
\label{T:P-graph}
Let $\Lambda$ be a $P$-graph such that $\Lambda^0$ has property \textsf{W}. Then the following told true. 
\begin{itemize}
\item[(i)] $\Per\Lambda$ is a subgroup of $\G(P)$.

\item[(ii)] If $\Lambda_{\Per}^0$ is non-empty, then it is a hereditary subset of $\Lambda^0$.

\item[(iii)] Let $\Bq: \G(P)\to \G(P)/{\Per\Lambda}$ be the quotient map. Then $\Lambda_{\Per}^0\Lambda/\sim$ is a $\Bq(P)$-graph with degree map $\tilde d:=\Bq\circ d$. 

\item[(iv)] $\Lambda_{\Per}^0\Lambda$ is isomorphic to the pullback $\Bq^*\Gamma$ via $\lambda\mapsto ([\lambda], d(\lambda))$.
\end{itemize}
\end{thm}

The group $\Per\Lambda$ is called the \textit{periodicity group} of $\Lambda$, and the above $\Bq(P)$-graph $\Lambda_{\Per}^0\Lambda/\sim$ 
is called the \textit{pushout of $\Lambda$}.

\subsection{The aperiodicity of $\Lambda/\sim$}
\label{SS:Aper}

By Theorem \ref{T:P-graph} (iii), for simplicity, from now on we assume that $\Lambda_{\Per}^0=\Lambda^0$. 
The main result of this subsection is that $\Gamma$ is aperiodic. 
Intuitively, this is very natural and simple: $\Gamma$ is obtained by 
removing \textit{all} periods of $\Lambda$, and so $\Gamma$ should have the trivial periodicity only.
 
\begin{thm}
\label{T:aper}
Let $\Lambda$ be a $P$-graph such that $\Lambda^0$ has property \textsf{W} and $\Lambda^0_{\Per}=\Lambda^0$. 
Then the push out $\Lambda/\sim$ is an aperiodic $\Bq(P)$-graph.
\end{thm}

\begin{proof}
Let $\Gamma=\Lambda/\sim$.
By Theorem \ref{T:P-graph} and Corollary \ref{C:tensor2}, it suffices to show that $\Per\Gamma=\{0\}$.

Let $\Bq: \G(P)\to \G(P)/\Per\Lambda$ be the quotient map, and identify $\Lambda$ with $\Bq^*\Gamma$ by Theorem \ref{T:P-graph} (iv).
Then completely similar to Proposition \ref{P:dotx} one can see that $\Bq$ induces the following homeomorphism:
\begin{align*}
\Bq_*: \Lambda^\infty & \to \Gamma^\infty, \\  
x&\mapsto \dot{x}:=\Bq_*(x): (0,\Bq(n))\mapsto [x(0,n)].
\end{align*} 
Also, the projection $\Bp_1: \Lambda\to\Gamma$ is now  given by
\[
\Bp_1(\lambda)=[\lambda] \qforal \lambda\in\Lambda.
\]

For convenience,
let $\Bq_2$ be the quotient map from $\Omega_P$ onto $\Omega_{\Bq(P)}$ defined by 
$\Bq_2(0,n)=(0, \Bq(n))$. 
Let $x\in\Lambda^\infty$ and note that,  for all $n\in P$,
one has $x(0,n)=\big(\Bp_1\big(x(0,n)\big), n\big)=\big([x(0,n)],n\big)$. 
Then one gets the following commuting diagram:
\[
		\xymatrix{
		\Omega_P \ar[r]^{x} \ar[d]_{\Bq_2} & \Lambda\ar[d]^{\Bp_1}&   \\
		\Omega_{\Bq(P)} \ar[r]_{\dot{x}} &\Gamma.  &
		}
\]

We now suppose that $\mu,\nu\in\Lambda$ such that $[\mu]$ and $[\nu]$ are equivalent in $\Gamma$:
\[
[\mu]\sim_\Gamma [\nu].
\]
In what follows, we show that $\mu$ and $\nu$ are actually equivalent in $\Lambda$:
\[
\mu\sim_\Lambda \nu.
\]
Once this is done, we have that $[\mu]=[\nu]$, which proves that $\Per \Gamma=\{0\}$. 

To this end, arbitrarily choose $x\in s(\mu)\Lambda^\infty$. Note that $s(\mu)=s(\nu)$ as $s([\mu])=s([\nu])$. 
So $\dot{x}\in s([\mu])\Gamma^\infty$. 
Thus one obtains the following consecutive implications: 
\begin{align*}
& \quad [\mu]\sim_\Gamma [\nu]\\
&\Rightarrow [\mu]\dot{x}=[\nu]\dot{x}\\
&\Rightarrow ([\mu]\dot{x})(\Bq_2(0,n))=([\nu]\dot{x})(\Bq_2(0,n))\\
&\quad \text{ (for all }n\in P \text{ such that }\Bq(n)-\tilde d([\mu]), \Bq(n)-\tilde d([\nu])\in \Bq(P))\\
&\Rightarrow \dot{x}\big(0,q(n)-\tilde d([\mu])\big)=\dot{x}\big(0,q(n)-\tilde d([\nu])\big) \\
&\Rightarrow \dot{x}\big(0,\tilde d([\nu])\big)=\dot{x}\big(0,\tilde d([\mu])\big) \text{ (by taking }\Bq(n)=\tilde d([\mu])+\tilde d([\nu]))\\
&\Rightarrow  [x\big(0,d(\nu)\big)]=[x\big(0,d(\mu)\big)]\quad (\text{by the definition of }\dot{x}). 
\end{align*}
Thus we have $d(\mu)-d(\nu)\in \Per\Lambda$ by the very definition of $\Per\Lambda$. So there is a unique $\nu'\in\Lambda$ such that 
\begin{align}
\label{E:sim}
\mu\sim \nu'  \quad\text{and}\quad d(\nu')=d(\nu)
\end{align}
by Lemma \ref{L:uniqueness} and our assumption $\Lambda_{\Per}^0=\Lambda^0$.
Hence $\mu x=\nu'x$, and so $[\mu]\dot x=[\nu']\dot x$. Combining this with  $[\mu]\dot x=[\nu]\dot x$ (see the first implication above) gives 
\[
[\nu]\dot x=[\nu']\dot x.
\]
But we have $\tilde d([\nu])=\tilde d([\nu'])$ from $d(\nu)=d(\nu')$. 
Applying Lemma \ref{L:uniqueness} to the $\Bq(P)$-graph $\Gamma$ yields $[\nu]=[\nu']$, and
to the $P$-graph $\Lambda$ again gives $\nu=\nu'$ as $d(\nu)=d(\nu')$. 
From \eqref{E:sim}, we prove $\mu\sim_\Lambda \nu$, as desired. 
\end{proof}


The following corollary is an immediate consequence of 
Theorem \ref{T:tensor} by letting $H:=\Per\Lambda$, and Theorems \ref{T:P-graph} and \ref{T:aper}.

\begin{cor}
\label{C:tensor2}
Let $\Lambda$ be a $P$-graph such that $\Lambda^0$ has property \textsf{W} and $\Lambda^0_{\Per}=\Lambda^0$. 
Then the pushout $\Lambda/\sim$ is an aperiodic $\Bq(P)$-graph, 
and furthermore we have the following embedding
\[
\pi: \ca(\Lambda)\hookrightarrow \ca(\Lambda/\sim) \otimes \ca({\Per\Lambda}),\  s_\lambda\mapsto s_{[\lambda]}\otimes V^{d(\lambda)-\jmath\circ d(\lambda)}.
\]  
Here, as before, $V_i$'s are unitary generators of $\ca(\Per\Lambda)$.   
\end{cor}

We should notice that the pullback and pushout processes are not the inverse of each other. 
Consider $1$-graph $E$ which consists of one vertex $v$ and one edge. Then its periodicity group is $\bZ$, and
accordingly, the pushout $\Gamma=E/\sim\, =\{v\}$. Combining Example \ref{Eg:1} gives 

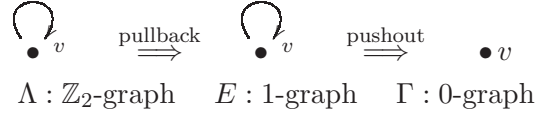
\begin{align*}
&\xymatrix{
\bullet\ar@(ul,ur)^>{v}
}
\quad \stackrel{\text{       pullback       }}\Longrightarrow\quad 
\xymatrix{
\bullet\ar@(ul,ur)^>{v}
}
\quad\stackrel{\text{       pushout       }}\Longrightarrow \quad
\xymatrix{
\bullet\, {v}
}\\
&{\Lambda: \bZ_2}\text{-graph }\quad  E: \text{1-graph } \quad \Gamma: \text{0-graph }
\end{align*}

However, unlike Example \ref{Eg:1}, one does obtain $\ca(E)\cong \ca(\Gamma)\otimes \ca(\Per E)$.  
Notice that, in Example \ref{Eg:1}, $H=2\bZ$ is a \textit{proper} subgroup of $\Per q^*\Gamma = \bZ$. 

So it seems more appropriate to expect the injection $\pi$ defined in Corollary \ref{C:tensor2} is surjective. But the following example, due to Aidan Sims, shows that 
our theorem \ref{T:tensor} is the best one could possibly get

\begin{eg}

\label{Eg:Sims}

Consider the $2$-graph $\Lambda$ drawn in Fig 1 below: 
\begin{align*}
&\xymatrix{
&&\bullet\ar@/^/[dd]^f^<{w}^>{v} 
\ar@/^/@{<--}[dll]^b\\
\bullet\ar@(ul,ur)^e
\ar@/^/@{<--}[urr]^a
\ar@/^/@{-->}[drr]^c\\
&&\bullet\ar@/^/[uu]^g  \ar@/^/@{-->}[ull]^d^>{u}
}
&
\xymatrix{
&&\bullet\ar@/^/@{.>}[dd]^f ^<{w}^>{v}
\ar@/^/@{<--}[dll]^b\\
\bullet\ar@{.>}@(ul,ur)^e
\ar@/^/@{<--}[urr]^a
\ar@/^/@{-->}[drr]^c\\
&&\bullet\ar@/^/@{.>}[uu]^g  \ar@/^/@{-->}[ull]^d^>{u}
}
\\
&\quad \text{Fig 1.~}\Lambda:\text{ a 2-graph} &\qquad \text{Fig 2.~}\Lambda/\sim: \text{ a }\bZ_2\times \bN\text{-graph}
\end{align*}
Then one can check that $\Per\Lambda=\{(2n, 0):n\in\bZ\}$, and 
\[
ee\sim u, \ fg\sim v,\  gf\sim w.
\] 
So $\Lambda/\sim$ is a $\bZ_2\times \bN$-graph as drawn in Fig 2.~above, which looks like $\Lambda$, but 
the dotted edges now have `torsion'.

We claim that there is \textit{no} any ``natural canonical" isomorphism between $\ca(\Lambda)$ and $\ca(\Lambda/\sim)\otimes \rC(\bT)$.
The idea is sketched as follows. 
Suppose that there is such an isomorphism $\pi$. 
On one hand, consider the $2$-cycle subgraph $C_2:=\{v,w,f,g\}$ in $\Lambda$. By \cite{Raeburn}, 
\[
\pi(s_f) = s_{[f]} \otimes z, \quad \pi(s_g) = s_{[g]} \otimes 1
\] 
(or the other way).
On the other hand,  since the dashed $1$-graph has no periodicity, it is reasonable to have 
\[
\pi(s_\alpha)= s_{[\alpha]} \otimes 1\quad \text{for all dashed paths } {\alpha}.
\]
Then one has 
\begin{align*}
\pi(s_e) 
&=\pi( s_e s^*_c s_c) 
  =\pi( s^*_d s_g s_c)= (s^*_{[d]} \otimes 1)(s_{[g]} \otimes 1)(s_{[c]} \otimes 1) \\ 
&= s^*_{[d]}s_{[g]}s_{[c]} \otimes 1 = s_{[e]} s^*_{[c]}s_{[c]} \otimes 1 = s_{[e]} \otimes 1,
\end{align*}
and 
\begin{align*}
\pi(s_e) 
&= \pi(s_e s^*_d s_d) 
   = \pi(s^*_c s_f s_d)  =(s^*_{[c]} \otimes 1)(s_{[f]} \otimes z)(s_{[d]} \otimes 1) \\ 
&= s^*_{[c]}s_{[f]}s_{[d]} \otimes z = s_{[e]} s^*_{[d]}s_{[d]} \otimes 1 = s_{[e]} \otimes z,
\end{align*}
a contradiction. 
\end{eg}


\section{A distinguished MASA}

\label{S:k-str}

In this section, we exhibit an interesting application of our results in Sections \ref{S:pullback} and \ref{S:PP}. 
But because of Theorem \ref{T:P-graph} (iv), we have to restrict ourselves to the following class of $k$-graphs $\Lambda$: 
\textbf{$\Lambda^0$ has property \textsf{W} and $\Lambda^0_{\Per}=\Lambda^0$.} 
This class includes a class of cofinal, sink-free $P$-graphs as special examples, 
which exhausts all strongly connected $k$-graphs studied in \cite{HLRS14} recently.
Our application answers the questions asked in \cite{BNR14} in this context. 

Let $\Lambda$ be a $k$-graph satisfying our assumptions. Let 
\[
\M:=\ca(s_\mu s_\nu^*: \mu\sim \nu),
\]
which is called the \textit{cycline subalgebra of $\ca(\Lambda)$}
in \cite{BNR14}. It turns out that $\M$ plays a central role in a generalized Cuntz-Krieger uniqueness theorem for $\Lambda$ \cite{BNR14}. It is shown that a representation of $\ca(\Lambda)$
is injective, if and only if, its restriction onto $\M$ is injective. Also, $\M$ is abelian and $\fD_\Lambda'=\M'$. 
Because of its importance, it is asked in \cite{BNR14} if $\M$ is a MASA and there is a faithful conditional 
expectation from $\ca(\Lambda)$ onto $\M$. In other words, it would be nice to know if $\M$ is an abelian core in the terminology of \cite{NR12}.
We will answer those questions affirmatively under our conditions. But a simple lemma first, which holds true for all $k$-graphs.

\begin{lem}
\label{L:5.5}
Let $\Lambda$ be an {\rm arbitrary} $k$-graph, and $\mu,\nu\in\Lambda$. 
Then $\mu\sim\nu$ if and only if, for any $n\le d(\mu)\land d(\nu)$, there are $w\in\Lambda^n$, $\mu', \nu'\in\Lambda$ such that 
$\mu=w \mu'$, $\nu=w\nu'$ and $\mu'\sim\nu'$. 
\end{lem}

\begin{proof}
It suffices to show ``only if" part as ``if" part is trivial.
Factor $\mu=\mu_1\mu_2$ and $\nu=\nu_1\nu_2$ with $d(\mu_1)=d(\nu_1)=n$. 
Let $y\in s(\mu)\Lambda^\infty$. Then 
\[
\mu_1\mu_2 y=\mu y=\nu y =\nu_1\nu_2y
\]
implies 
$(\mu_1\mu_2 y) (0,n)=(\nu_1\nu_2y)(0,n)$, i.e., $\mu_1=\nu_1=:w$, as $d(\mu_1)=d(\nu_1)$. 
Also 
$\sigma^n(\mu_1\mu_2 y)=\sigma^n(\nu_1\nu_2y)$ 
implying $\mu_2 y=\nu_2 y$.
Letting $\mu'=\mu_2$ and $\nu'=\nu_2$ ends the proof. 
\end{proof}

\begin{thm}
\label{T:masa}
Suppose $\Lambda$ is a $k$-graph such that $\Lambda^0$ has property \textsf{W} and $\Lambda^0_{\Per}=\Lambda^0$. 
Then we have the following. 
\begin{itemize}
\item[(i)] 
$\M$ is a MASA in $\ca(\Lambda)$. Furthermore, 
\[
\M=\fD_\Lambda' \cong \fD_\Lambda\otimes \ca(\Per\Lambda).
\]
\item[(ii)] There is a faithful conditional expectation from $\ca(\Lambda)$ onto $\M$.
\end{itemize}
\end{thm}

\begin{proof}
(i) 
If $\Lambda$ is aperiodic, then $\fD_\Lambda$ is a MASA by \cite[Theorem]{Hop05}. 
Also, by Theorem \ref{T:cha}, one has $\M=\ca(s_\mu s_\mu^*:\mu\in\Lambda)=\fD_\Lambda$. 
Then the conclusion follows from Theorem \ref{T:cha}.

We now assume that $\Lambda$ is periodic. So $\Per\Lambda$ is a non-zero subgroup of $\bZ^k$.  
Let $\Gamma:=\Lambda/\sim$ and $\Bq:\bZ^k\to \bZ^k/\Per\Lambda$ be the quotient map.
Then it follows from Theorem \ref{T:P-graph} that  
$\Lambda$ is isomorphic to the pullback $\Bq^*\Gamma$.
Keeping the same notation as in Corollary \ref{C:tensor2},
we have the following embedding
\begin{align*}
\pi: \ca(\Lambda)\cong\ca(\Bq^*\Gamma) &\hookrightarrow \ca(\Gamma) \otimes \ca(\Per\Lambda),\\
s_\lambda &\mapsto s_{[\lambda]}\otimes V^{d(\lambda)-\jmath\circ d(\lambda)}.
\end{align*}
Then one can check that 
\begin{align*}
\fB:=\pi(\ca(\Bq^*\Gamma))
&=\ca\big(s_\mu s_\nu^* \otimes V^n: d(\mu)-d(\nu)=q(n-\jmath\circ q(n)\big)\\
&=\ca\big(s_\mu s_\nu^* \otimes V^n: d(\mu)_i-d(\nu)_i=[n_i]_{d_i}: 1\le i\le r\big).
\end{align*}

Since $\Gamma$ is aperiodic by Theorem \ref{T:aper},
the diagonal algebra $\fD_\Gamma$ of $\ca(\Gamma)$ is a MASA as shown completely similar to \cite[Theorem]{Hop05}.
So 
$\fD_\Gamma'\cap \big(\ca(\Gamma)\otimes \ca(\Per\Lambda)\big)=\fD_\Gamma\otimes \ca(\Per\Lambda)$ is a MASA of $\ca(\Gamma)\otimes \ca(\Per\Lambda)$ \cite{Was76, Was08}. 
It is now not hard to verify that $\fD_\Gamma'\cap \fB$ is $\fD_\Gamma\otimes \ca\big (V_1^{d_1}, \ldots, V_r^{d_r}\big)$, and 
it is a MASA of $\fB$. Moreover, its pre-image 
\begin{align*}
\pi^{-1}(\fD_\Gamma'\cap\fB)
&=\pi^{-1}\big(\fD_\Gamma\otimes \ca\big (V_1^{d_1}, \ldots, V_r^{d_r}\big)\big)\\
&=\fD_\Lambda'\cap\ca(\Lambda)=\fD_\Lambda\ca(W_h: h\in \Per\Lambda)\\
&\cong \fD_\Lambda\otimes \ca(\Per\Lambda)
\end{align*}
by Corollary \ref{C:tensor}. (Recall from Lemma \ref{L:W} that, for any $h\in \Per\Lambda$, $W_h$ is a central unitary in $\M(\ca(\Lambda))$.)
Thus $\fD_\Lambda'\cap\ca(\Lambda)$ is a MASA of $\ca(\Lambda)$, which is isomorphic to $\fD_\Lambda\otimes \ca(\Per\Lambda)$. 
 
To finish proving (i), we now show that $\M=\fD_\Lambda'\cap\ca(\Lambda)$, namely 
\[
\M=\fD_\Lambda\ca(W_h: h\in \Per\Lambda).
\] 
Obviously, the right hand side is contained in the left hand side. To show the other inclusion, we make use of Lemma 
\ref{L:5.5}. Let $\mu\sim \nu$. Then $(\mu,\nu)=(w\mu', w\nu')$ with $d(w)=d(\mu)\wedge d(\nu)$ and $\mu'\sim \nu'$. 
Note $d(\mu')\land d(\nu')=0$. Let $h=d(\mu')-d(\nu')$.
It follows from Lemma \ref{L:W} that
$s_{([\lambda], d(\mu'))}=W_h \, s_{([\lambda],d(\nu'))}$.
In particular, this implies $s_{\mu'}=W_hs_{\nu'}$. Hence 
\[
s_\mu s_\nu^*=W_hs_ws_{\nu'}s_{\nu'}^*s_w^*=s_{w\nu'}s_{w\nu'}^* W_h\in\fD_\Lambda\ca(W_h: h\in \Per\Lambda).
\] 

\smallskip
(ii) This immediately follows from Corollary \ref{C:tensor}, Corollary \ref{C:tensor2} and (i) above, 
as $\fD_\Gamma$ is the canonical MASA of the AF-algebra $\fF_\Gamma$,
the fixed point algebra of the gauge action of $\ca(\Gamma)$. 
\end{proof}

The fact that $\M$ is  a MASA of $\ca(\Lambda)$ can also follow the facts $\M=\fD_\Lambda'$ proved above and $\M'=\fD_\Lambda'$ shown 
in \cite{BNR14}. 

Combining Theorem \ref{T:masa} and \cite{BNR14} gives 

\begin{cor}
\label{C:expectation}
Under the same conditions as in Theorem \ref{T:masa}, the cycline subalgebra $\M$ is an abelian core of $\ca(\Lambda)$. 
\end{cor}

Let us end the paper by remarking the simplicity of $\ca(\Gamma)$ and the centre of $\ca(\Lambda)$, which 
generlize \cite[Corollary 8.8]{DY09b} and \cite[Proposition 6.1]{HLRS14}. 

\begin{cor}
Keep the same conditions as in Theorem \ref{T:masa}.  Suppose further that $\Lambda$ is finite and cofinal. 
Then $\ca(\Lambda/\sim)$ is simple and the centre of $\ca(\Lambda)$ is $\ca(W_h: h\in \Per\Lambda)$. 
\end{cor}

\begin{proof}
Clearly $\Gamma$ is finite and so $\ca(\Gamma)$ is unital.
By Proposition \ref{P:dotx}, it is easy to see that $\Gamma$ is also cofinal. By Theorem \ref{T:aper}, $\Lambda/\sim$ is aperiodic. As \cite[Proposition 4.8]{KumPask}, one can see that
$\ca(\Lambda/\sim)$ is simple. Thus $\Z(\ca(\Lambda))=\pi^{-1}(\Z(\pi(\ca(\Lambda))))=\pi^{-1}(I\otimes \ca(V_1^{d_1},\ldots, V_r^{d_r}))=\ca(W_1,\ldots, W_r)$ from \eqref{E:1}.
\end{proof}

\subsection*{Acknowledgements} The author would like to heartily thank Professor Aidan Sims for pointing out a mistake in the proof of 
Theorem \ref{T:tensor} in an earlier version, and for providing very valuable suggestions and comments. 
Also, the author is very grateful  for Professor Ken Davidson's useful comments and encouragement.

\end{document}